\documentclass{article}

\usepackage[english]{babel}

\usepackage[letterpaper,top=2cm,bottom=2cm,left=3cm,right=3cm,marginparwidth=1.75cm]{geometry}

\usepackage{amssymb,mathrsfs,amscd}
\usepackage{amsthm}
\usepackage{hyperref}
\usepackage{amsmath}
\usepackage{lineno,hyperref,color}
\usepackage{amsmath}
\usepackage{graphicx}

\title{Shadowing, average shadowing and transitive properties of multiple mappings}
\author{Yingcui Zhao}

\begin{document}
\newtheorem{theorem}{Theorem}[section]
\newtheorem{corollary}[theorem]{Corollary}
\newtheorem{lemma}[theorem]{Lemma}
\newtheorem{proposition}[theorem]{Proposition}
\newtheorem{problem}[theorem]{Problem}
\newtheorem{maintheorem}[theorem]{Main Theorem}
\newtheorem{definition}[theorem]{Definition}
\newtheorem{remark}[theorem]{Remark}
\newtheorem{example}[theorem]{Example}
\newtheorem{claim}{Claim}[section]
\maketitle

\begin{abstract}
In 2016, Hou and Wang\cite{HouWang2016} introduced the concept of multiple mappings based on iterated function system, which is an important branch of fractal theory. In this paper, we introduce the definitions of shadowing, average shadowing, transitive, weakly mixing, mixing, chain transitive and chain mixing properties of multiple mappings from a set-valued perspective. We show that both two continuous self-maps have shadowing (respectively, average shadowing, transitive, weakly mixing, mixing) property may not imply the multiple mappings they form has the corresponding property. While a sufficient condition for multiple mappings to be transitive (respective, weakly mixing, mixing, chain transitive, chain mixing) is given. Both shadowing and average shadowing properties are invariant under the iterative action of multiple mappings. Also, we study that for multiple mappings shadowing property plus chain mixing implies mixing and average shadowing properties implies chain transitivity.
\end{abstract}

\section{Introduction}
A dynamical system typically refers to the pair $(X,f)$, in which $X$ is a compact metric space with a metric $d$ and $f$ is a continuous self-map on $X$. Shadowing, average shadowing and transitive properties are all very important concepts in dynamical systems. They are closely related to the stability and chaotic behavior of a dynamical system. By studying these properties of a dynamical system, we can understand its connectivity and transitive properties, revealing the interrelations and interactions within the system. They can also be applied to various practical problems, such as traffic flow, financial markets, ecological systems, and more. By analyzing them, it can help understand and predict the behavior of these complex systems, providing guidance and methods for solving practical problems. 

Let $\mathbb{N}=\{0,1,2,\cdots\}$, $\mathbb{Z^+}=\{1,2,3,\cdots\}$. 
A sequence $\{x_n\}_{n\geq0}\subset X$  is said to be \emph{a $\delta$-pseudo orbit of $f$}, if for any $n\in\mathbb{N}$, $d(f(x_n),x_{n+1})\leq\delta$ (when $n$ is a finite number, we call $\{x_n\}_{n\geq0}\subset X$ as \emph{a $\delta$-chain}). We say $f$ has \emph{shadowing property}, if for any $\epsilon>0$ there exists $\delta>0$ such that every $\delta$-pseudo orbit of $f$ is $\epsilon$-shadowed by the orbit $orb(y,f):=\{f^n(y)\mid  n\in\mathbb{N}\}$ of some point $y$ in $X$ under $f$, i.e., for any $n\in\mathbb{N}$, $d(f^n(y),x_n)<\epsilon$.

A sequence $\{x_n\}_{n\geq0}\subset X$ is said to be \emph{a $\delta$-average-pseudo-orbit of $f$}, if there is $N(\delta)>0$ such that for any $n\geq N(\delta)$ and any $k\in\mathbb{Z^+}$, $$\frac{1}{n}\sum^{n-1}_{i=0}d(f(x_{i+k}),x_{i+k+1})<\delta.$$
We say $f$ has \emph{average shadowing property}, if for any there exists $\delta>0$ such that every $\delta$-avarage-pseudo-orbit of $f$ is $\epsilon$-average-shadowed by the orbit $orb(y,f):=\{f^n(y)\mid  n\in\mathbb{N}\}$ of some point $y$ in $X$ under $f$, i.e., $$\limsup_{n\rightarrow\infty}\frac{1}{n}\sum^{n-1}_{i=0}d(f^i(y),x_i)<\epsilon.$$

We say $f$ is
\begin{enumerate}	
\item[(1)] \emph{transitive}, if for any nonempty open sets $U,V\subset X$, there exists $n\in\mathbb{Z^+}$ such that $f^n(U)\bigcap V\neq\emptyset$.

\item[(2)]\emph{weakly mixing}, if for any nonempty open sets $U_1,V_1,U_2,V_2\subset X$, there exists $n\in\mathbb{Z^+}$ such that $f^n(U_1)\bigcap V_1\neq\emptyset$ and $f^n(U_2)\bigcap V_2\neq\emptyset$.

\item[(2)]\emph{mixing}, if for any nonempty open sets $U,V\subset X$, there is $N\in\mathbb{Z^+}$ such that for any $n\geq N$, $f^n(U)\bigcap V\neq\emptyset$.

\item[(3)]\emph{chain transitive}, if for any $x,y\in X$ and $\delta>0$ there is a $\delta$-chain from $x$ to $y$.

\item[(4)]\emph{chain mixing}, if for any $x,y\in X$ and $\delta>0$ there is $N\in\mathbb{Z^+}$ such that for any $n\geq N$, one can find a $\delta$-chain from $x$ to $y$ of length $n$.
\end{enumerate}

In 2016, Hou and Wang\cite{HouWang2016} defined multiple mappings derived from iterated function system. Iterated function system is an important branch of fractal theory, reflecting the essence of the world. They are one of the three frontiers of nonlinear science theory. Hou and Wang's focus was primarily on studying the Hausdorff metric entropy and Friedland entropy of multiple mappings. Additionally, they introduced the notions of Hausdorff metric Li-Yorke chaos and Hausdorff metric distributional chaos from a set-valued perspective. It is worth noting that researchers studying iterated function systems often approach the topic from a group perspective rather than a set-valued perspective. This also establishes a close connection between multiple mappings and set-valued mappings, or we can consider multiple mappings as a special case of set-valued mappings.

It is important to acknowledge the valuable role of set-valued mappings in addressing complex problems involving uncertainty, ambiguity, or multiple criteria. Set-valued mappings offer versatility and flexibility, making them highly beneficial across various fields. One prominent application of set-valued mappings is in optimization problems, where the objective is to identify the optimal set of solutions. For instance, in multi-objective optimization, a set-valued mapping can represent the Pareto front, encompassing all non-dominated solutions.

Set-valued mappings also prove useful in decision-making processes that require considering multiple criteria or preferences. By representing feasible solutions as sets, decision-makers can thoroughly analyze and compare different options, enabling them to make well-informed decisions. Additionally, set-valued mappings find applications in data analysis tasks such as clustering and classification. Unlike assigning each data point to a single category, set-valued mappings can represent uncertainty or ambiguity by assigning data points to multiple categories. In fact, the applications of set-valued mappings are vast and diverse, encompassing numerous fields beyond those mentioned here.

In \cite{ziji}, we studied that if multiple mappings $F$ has a disdributionally chaotic
pair, especially $F$ is distributionally chaotic, $S\Omega(F)$
contains at least two points and gives a sufficient condition for
$F$ to be distributionally chaotic in a sequence and chaotic in the
strong sense of Li-Yorke. Zeng et al. \cite{zeng2020} proved two topological conjugacy dynamical systems to multiple mappings have simultaneously Hausdorff metric Li-Yorke chaos or Hausdorff distributional chaos and the multiple mappings $F=\{f_1,f_2\}$ and its $2$-tuple of continuous self-maps $f_1,f_2$ are not mutually implied in terms of Hausdorff metric Li-Yorke chaos.

The current paper aims to consider the image of one point under multiple mappings as a set (a compact set). We primarily consider the shadowing, average shadowing and transitive properties of multiple mappings. The specific layout of the present paper is
as follows. Some preliminaries and definitions are introduced in Section \ref{sec2}.
Then we study the relation between multiple mappings and its continuous self-maps in terms of shadowing, average shadowing and transitive properties in Section \ref{sec3}, the iterability of shadowing and average shadowing properties in Section \ref{sec4}, and the relationships among transitive, weakly mixing, mixing, chain transitive and chain mixing in Section \ref{sec5}. At last, we summarize and state the conclusions of this paper in Section \ref{sec6}.

\section{Preliminaries}\label{sec2}
Let $F=\{f_1,f_2,\cdots,f_n\}$ be a \emph{multiple mappings} with $n$-tuple of continuous self-maps on $X$. Note that for any $x\in X$, $F(x)=\{f_1(x),f_2(x),\cdots,f_n(x)\}\subset X$ is compact. Let

$$\mathbb{K}(X)=\{K\subset X\mid  K is nonempty and compact\}.$$ 
Then $F$ is from $X$ to $\mathbb{K}(X)$. The Hausdorff metric $d_H$ on $\mathbb{K}(X)$ is defined by $$d_H(A,B)=\max\{\sup_{a\in A}\inf_{b\in B}d(a,b),\sup_{b\in B}\inf_{a\in A}d(a,b)\},\forall A,B\subset X.$$ It is clear that $\mathbb{K}(X)$ is a compact metric space with the Hausdorff metric $d_H$.

For convenience, let us study multiple mappings on compact metric spaces from a set-valued view, using the example of examining two continuous self-maps. The definitions and conclusions presented in this paper can be easily extended to the case of a multiple mappings formed by any finite number of continuous self-maps. Consider the multiple mappings $F=\{f_1,f_2\}$. For any $n>0$, $F^n:X\rightarrow\mathbb{K}(X)$ is defined by for  any $x\in X$,$$F^n(x)=\{f_{i_1}f_{i_2}\cdots f_{i_n}(x)\mid  i_1,i_2,\cdots,i_n=1\text{ or }2\}.$$
Particularly, if $A\in\mathbb{K}(X)$, $F^n(A)\in\mathbb{K}(X)$. Then, $F$ naturally induces a continuous self-map on $\mathbb{K}(X)$, denoted by $\widetilde{F}:\mathbb{K}(X)\rightarrow\mathbb{K}(X)$.

Throughout this paper, 	Suppose that $F=\{f_1,f_2\}$ be a multiple mappings with $2$-tuple of continuous self-maps on $X$. Now, we define the concept of shadowing property of multiple mappings $F=\{f_1,f_2\}$ from a set-valued view. 


\begin{definition}\label{pfp}
A sequence $\{A_n\}_{n\geq0}\subset \mathbb{K}(X)$ with $A_0$ being a singleton set is said to be a $\delta$-pseudo orbit of $F$, if for any $n\in\mathbb{N}$, $d_H(F(A_n),A_{n+1})\leq\delta$. We say $F$ has (Hausdorff metric) shadowing property, if for any $\epsilon>0$ there exists $\delta>0$ such that every $\delta$-pseudo orbit of $F$ is $\epsilon$-shadowed by the orbit $orb(y,F):=\{F^n(y)\mid  n\in\mathbb{N}\}$ of some point $y$ in $X$ under $F$, i.e., for any $n\in\mathbb{N}$, $d_H(F^n(y),A_n)<\epsilon$.
\end{definition}

\begin{definition}\label{apfp}
A sequence $\{A_n\}_{n\geq0}\subset \mathbb{K}(X)$ with $A_0$ being a singleton set is said to be a $\delta$-average-pseudo orbit of $F$, if for any $n\in\mathbb{N}$, $d_H(F(A_n),A_{n+1})\leq\delta$. We say $F$ has (Hausdorff metric) shadowing property, if there is $N(\delta)>0$ such that for any $n\geq N(\delta)$ and any $k\in\mathbb{Z^+}$, $$\frac{1}{n}\sum^{n-1}_{i=0}d_H(F(A_{i+k}),A_{i+k+1})<\delta.$$
We say $F$ has (Hausdorff metric) average shadowing property, if for any there exists $\delta>0$ such that every $\delta$-avarage-pseudo-orbit of $F$ is $\epsilon$-average-shadowed by the orbit $orb(y,F):=\{F^n(y)\mid  n\in\mathbb{N}\}$ of some point $y$ in $X$ under $F$, i.e., $$\limsup_{n\rightarrow\infty}\frac{1}{n}\sum^{n-1}_{i=0}d_H(F^i(y),A_i)<\epsilon.$$
\end{definition}

Let $ Ran(F)=\{F^n(x)\mid  n\geq 1,x\in X\}$.
\begin{definition}\label{newde2}
We say $F$ is
\begin{enumerate}	
	\item[(1)] \emph{transitive}, if for any nonempty open sets $U\subset X$ and $\mathcal{U}\subset Ran(F)$, there exists $n\in\mathbb{Z^+}$ such that $$\{F^n(u)\mid  u\in U\}\bigcap\mathcal{U}\neq\emptyset.$$
	
	\item[(2)]\emph{weakly mixing}, if for any nonempty open sets $U,V\subset X$ and $\mathcal{U},\mathcal{V}\subset Ran(F)$, there exists $n\in\mathbb{Z^+}$ such that $$\{F^n(u)\mid  u\in U\}\bigcap\mathcal{U}\neq\emptyset~and~\{F^n(v)\mid  v\in V\}\bigcap\mathcal{V}\neq\emptyset.$$
	
	\item[(2)]\emph{mixing}, if for any nonempty open sets $U\subset X$ and $\mathcal{U}\subset Ran(F)$,  there is $N\in\mathbb{Z^+}$ such that for any $n\geq N$, $$\{F^n(u)\mid  u\in U\}\bigcap\mathcal{U}\neq\emptyset.$$
	
	\item[(3)]\emph{chain transitive}, if for any $x\in X$, $A\in Ran(F)$ and $\delta>0$ there is a $\delta$-chain from $x$ to $A$.
	
	\item[(4)]\emph{chain mixing}, if for any $x\in X$, $A\in Ran(F)$ and $\delta>0$ there is $N\in\mathbb{Z^+}$ such that for any $n\geq N$, one can find a $\delta$-chain from $x$ to $A$ of length $n$.
\end{enumerate}
\end{definition}
Clearly, the Hausdorff metric shadowing, average shadowing and transitive properties of multiple mappings, in the case of degradation (where the multiple mappings consists of only one continuous self-map), is the same as the shadowing, average shadowing and transitive properties of a classical single continuous self-map. Next we provide an example to illustrate the existence of the newly defined concept Definition \ref{pfp}, Definition \ref{apfp} and Definition \ref{newde2}. 

\begin{example}\label{eab}
Consider the multiple map defined on a compact space $X$ as $F=\{f_1,f_2\}$, in which both $f_1$ and $f_2$ are constant functions. Let $f_1(x)\equiv a$ and $f_2(x)\equiv b$. It is can be verified that for any $x\in X$ and any $n\in\mathbb{N}$, $F^n(x)=\{a,b\}$. Let $\epsilon>0$ and $\delta=\frac{\epsilon}{2}>0$. Then for every $\delta$-pseudo orbit $\{A_n\}_{n\geq0}$ of $F$, $$d_H(\{a,b\},A_{n})\leq\delta, \forall n\in\mathbb{Z^+}.$$ Thus for any $y\in X$ and any $n\in\mathbb{N}$, $$d_H(F^n(y),A_n)\leq\delta<\epsilon.$$
So, $F$ has shadowing and average shadowing properties. 
\end{example}

\begin{example}\label{e1}
Consider the multiple map defined on $[0,1]$ as $F=\{f_1,f_2\}$, in which $f_1(0)=f_1(1)=0$ and $f_2(0)=f_2(1)=1$. Then $Ran(F)=\{\{0,1\}\}$ and $F(0)=\{0,1\}=F(1)$. So, $F$ possesses all the topological properties defined in Definition \ref{newde2}.
\end{example}

%
\section{Relation Between $F$ and $f_1,f_2$}\label{sec3}
A natural question is what is the implication between the shadowing and average shadowing properties of multiple mappings $F=\{f_1,f_2\}$ and the shadowing and average shadowing properties of its $2$-tuple of continuous self-maps $f_1,f_2$? 
The next example show that both $f_1$ and $f_2$ have shadowing (respectively, average shadowing) property may not imply the shadowing (respectively, average shadowing) property of $F$.
\begin{example}
Consider the multiple mappings defined on $[0,1]$ as $F=\{f_1,f_2\}$, in which $f_1(x)=0,\forall x\in[0,1]$ and 	
\begin{equation*}
	f_2(x)=\begin{cases}
		2x, & 0\leq x\leq\frac{1}{2}, \\
		2-2x, & \frac{1}{2}<x\leq 1.
	\end{cases}
\end{equation*}
\begin{itemize}
	\item [(1)]As we all know, both $f_1$ and $f_2$ have shadowing and average shadowing property.
	\item[(2)]Let $\epsilon=\frac{1}{49}>0$ and $\delta>0$. 
	Then there exists $\delta$-pseudo orbit $$\{\{\frac{1}{7}\},\{\frac{2}{7},\delta\},\{\frac{4}{7},\delta,3\delta\},\{\frac{6}{7},\delta,3\delta,7\delta\},\{\frac{2}{7},\delta,3\delta,7\delta,15\delta\},\cdots,A,\cdots\},$$
	in which $A\subset [0,1]$ with $a\in A$ satisfying $a>\frac{1}{2}$. Clearly, it is also a $\delta$-average-pseudo-orbit of $F$. While, for any $z\in [0,1]$ and any $n>0$, $d_H(F^n(z),A)=d_H(\{0,f_2^n(z)\},A)>\epsilon$. So, $F$ doesn't have shadowing or average shadowing property.
\end{itemize}
\end{example}

\begin{theorem}
If $f_1(x)=c$ ($\forall x\in X$, $c$ is a constant) , $f_2(c)=c$, and $F$ has shadowing (respectively, average shadowing) property, then $f_2$ has shadowing (respectively, average shadowing) property.
\end{theorem}
\begin{proof}
For any $x\in X$ and any $n>$, $F^n(x)=\{c,f_2^n(x)\}$. Let $\epsilon>0$ and $\delta>0$ be such that $F$ has shadowing property witnessed by $\frac{\epsilon}{2}$. Let $\{x_i\}_{i\geq 0}$ be a $\delta$-pseudo orbit of $f_2$. Then $\{c,x_i\}_{i\geq 0}$ is a $\delta$-pseudo orbit of $F$. Thus there exists $z\in X$ such that $$d_H(\{c,x_n\},\{c,f_2^n(z)\})<\frac{\epsilon}{2}.$$ Then $d(x_n,f_2^n(z))<\epsilon$. So, $f_2$ has shadowing property.

The proof regarding the average shadowing property is similar.
\end{proof}

The Example \ref{e1} shows that the transitivity / weakly mixing / mixing / chain transitivity / chain mixing of multiple mappings $F=\{f_1,f_2\}$ doesn't imply the  transitivity / weakly mixing / mixing / chain transitivity / chain mixing of $f_1$ or $f_2$. Then combined with the following example, the transitivity / weakly mixing / chain transitivity of multiple mappings $F=\{f_1,f_2\}$ and its $2$-tuple of continuous self-maps $f_1,f_2$ do not imply each other.
\begin{example}
Consider the multiple mappings  $F=\{f_1,f_2\}$ defined on $\{0,1,2\}$, in which $f_1:0\longmapsto1\longmapsto2\longmapsto0$ and $f_2:0\longmapsto2\longmapsto1\longmapsto0$. It is easy to see that both $f_1$ and $f_2$ is transitive, weakly mixing and chain transitive. Next we show $F$ is not chain transitive.

$Ran(F)=\{\{1,2\},\{0,2\},\{0,1\},\{0,1,2\}\}$. Let $x=0$ and $A=\{0,2\}\in Ran(F)$. While $F:0\mapsto\{1,2\}\mapsto\{0,1,2\}\mapsto\cdots\mapsto\{0,1,2\}\mapsto\cdots$. So, $F$ is not chain  transitive. Then $F$ is not transitive or weakly mixing.
\end{example}

%

Although that both $f_1$ and $f_2$ are transitive / weakly mixing / chain transitive can't imply $F=\{f_1,f_2\}$ is transitive / weakly mixing / chain transitive, now we provide a sufficient condition for $F$ to be transitive / weakly mixing / chain transitive as Theorem \ref{thFtra}, also containing a sufficient condition for $F$ to be chain nixing / mixing.  
\begin{theorem}\label{thFtra}
If $f_1$ is a constant function and $f_2$ is transitive (respectively, weakly mixing, mixing, chain transitive, chain mixing), then $F$ is transitive (respectively, weakly mixing, mixing, chain transitive, chain mixing).
\end{theorem}
\begin{proof}
Let $f_1(x)=c$ ($\forall x\in X$, $c$ is a constant). Then it is esay to see that $Ran(F)=\{\{c,f_2^n(x)\}\mid  n\geq 1, x\in X\}$. 
\begin{itemize}
	\item [(1)]Transitive: Let $U\subset X$ and $\mathcal{U}\subset Ran(F)$ be two nonempty open sets. Then there exists nonempty open set $V\subset X$ such that $\{\{c,v\}\mid  v\in V\}\subset\mathcal{U}$. Since $f_2$ is transitive, there exists $n\in\mathbb{Z^+}$ such that $f_2^n(U)\bigcap V\neq\emptyset$. Then there exists $u\in U$ such that $f_2^n(u)\in V$. Hence $F^n(u)\in\mathcal{U}$. So, $F$ is transitive.
	
	\item[(2)]Weakly mixing: Let $U,V\subset X$ and $\mathcal{U},\mathcal{V}\subset Ran(F)$ be nonempty open sets. Then there exists nonempty open set $U_1,V_1\subset X$ such that $\{\{c,x\}\mid  x\in U_1\}\subset\mathcal{U}$ and $\{\{c,x\}\mid  x\in V_1\}\subset\mathcal{V}$. Since $f_2$ is weakly mixing, there exists $n\in\mathbb{Z^+}$ such that $f_2^n(U)\bigcap U_1\neq\emptyset$ and $f_2^n(V)\bigcap V_1\neq\emptyset$. Then there exist $u\in U$ and $v\in V$ such that $f_2^n(u)\in U_1$ and $f_2^n(v)\in V_1$. Hence, $F^n(u)\in\mathcal{U}$ and $F^n(v)\in\mathcal{V}$. So, $F$ is weakly mixing. 
	
	\item[(3)]Mixing: Let $U\subset X$ and $\mathcal{U}\subset Ran(F)$ be two nonempty open sets. Then there exists nonempty open set $V\subset X$ such that $\{\{c,v\}\mid  v\in V\}\subset\mathcal{U}$. Since $f_2$ is mixing, there exists $N>0$ such that there exists $u\in U$ such that $f_2^n(u)\in V$ for any $n\geq N$. Hence $F^n(u)\in\mathcal{U}$ for any $n\geq N$. So, $F$ is mixing.
	
	\item[(4)]Chain transitive: Let $x\in X$, $A\in Ran(F)$ and $\delta>0$. Then $A=\{c,f_2^l(a)\}$ for some $l>0$ and some $a\in X$. Since $f_2$ is chain transitive, there is a $\delta$-chain from $x$ to $f_2^l(a)$. Denote this $\delta$-chain by $\{x_,x_1,x_2,\cdots,x_n,f_2^l(a)\}$. Then $$\{x,\{x_1,c\},\{x_2,c\},\cdots,\{x_n,c\},A\}$$ must be a $\delta$-chain from $x$ to $A$.
	
	\item[(5)]Chain mixing: Let $x\in X$, $\{c,f_2^l(a)\}\in Ran(F)$ and $\delta>0$, where $l>0$, $a\in X$. Since $f_2$ is chain mixing, there is $N>0$ such that one can find a $\delta$-chain from $x$ to $f_2^l(a)$ of length $n$ for any $n\geq N$. Then there  must be a $\delta$-chain from $x$ to $\{c,f_2^l(a)\}$ of length $n$ for any $n\geq N$. So, $F$ is chain mixing.
\end{itemize}
\end{proof}

\section{Iterability of shadowing and average shadowing properties}\label{sec4}
Both shadowing and average shadowing properties are invariant under iterations of continuous self-maps on a compact metric space\cite{aoki1994,niu2011}. Inspired by this, it is natural for us to ask: Do the two properties have similar conclusions under the action of multiple mappings? Now we prove that the answer to this question is affirmative, but it is important to note that the proof process is not identical.
\begin{theorem}
The following statements are equivalent:
\begin{itemize}
	\item [(1)]$F$ has shadowing property.
	\item[(2)]$F^k$ has shadowing property for any $k>1$.
	\item[(3)]$F^k$ has shadowing property for some $k>1$.
\end{itemize}
\end{theorem}
\begin{proof}
$(1)\Rightarrow(2):$ Let $k>1$, $\epsilon>0$ and $\delta>0$ be such that $F$ has shadowing property witnessed by $\epsilon$. Let $\{A_n\}_{n\geq 0}$ be a $\delta$-pseudo orbit of $F^k$. Then for any $n\in\mathbb{N}$, $d_H(F^k(A_n),A_{n+1})<\delta$. Define $\{B_n\}_{n\geq 0}$ as $B_{mk+r}=F^r(A_m)$ for any $m\geq 0$ and any $0\leq r<k$. That is $$\{B_n\}_{n\geq 0}=\{A_0,F(A_0),\cdots,F^{k-1}(A_0),A_1,F(A_1),\cdots,F^{k-1}(A_1),A_2,\cdots,\}.$$ It is a $\delta$-pseudo orbit of $F$. Then there exists $z\in X$ such that for any $n\in\mathbb{N}$, $d_H(F^n(z),B_n)<\epsilon$. Thus $$d_H(F^{nk}(z),A_n)<\epsilon,\forall n\in\mathbb{N}.$$ So, $F^k$ has shadowing property.

$(2)\Rightarrow(3)$ is obvious.

$(3)\Rightarrow(1):$  Let $\epsilon>0$. Since $\widetilde{F}:\mathbb{K}(X)\rightarrow\mathbb{K}(X)$ is continuous and $\mathbb{K}(X)$ is compact with $d_H$, then by [\cite{lem}, Lemma 1.2] there exists $0<\gamma<\frac{\epsilon}{2}$ such that for arbitrary length $\gamma$-chain of length $k+1$ and any $z\in B_d(A_0,\gamma)$, $$d_H(F^i(z),A_i)<\epsilon,\forall 0\leq i\leq k.$$ Let $0<\eta<\gamma$ be such that $F^k$ has shadowing property witnessed by $\gamma$. Since $F$ is continuous and $X$ is compact, there exists $0<\delta<\min\{\frac{\eta}{k+1},\gamma\}$ such that $d(u,v)<\delta$ implies $$d_H(F^s(u),F^s(v))<\frac{\eta}{k+1},\forall 0\leq s\leq k.$$ 

Now we show $\delta$ makes $F$ have shadowing property witnessed by $\epsilon$.

Let $\{A_i\}_{i\geq 0}$ be a $\delta$-pseudo orbit of $F$. Then we show $\{A_{ik}\}_{i\geq 0}$ be a $\eta$-pseudo orbit of $F^k$. For any $i\geq 0$, 
\begin{equation*}
	\left\{
	\begin{aligned}
		& d_H(F(A_{ik}),A_{ik+1})<\delta\\
		& d_H(F(A_{ik+1}),A_{ik+2})<\delta \\
		& d_H(F(A_{ik+2}),A_{ik+3})<\delta\\
		&\cdots\\
		& d_H(F(A_{ik+k-2}),A_{ik+k-1})<\delta\\
		& d_H(F(A_{ik+k-1}),A_{ik+k})<\delta\\
	\end{aligned}
	\right.
\end{equation*}
Then,
\begin{equation*}
	\left\{
	\begin{aligned}
		& d_H(F^k(A_{ik}),F^{k-1}(A_{ik+1}))<\frac{\eta}{k+1}\\
		& d_H(F^{k-1}(A_{ik+1}),F^{k-2}(A_{ik+2}))<\frac{\eta}{k+1}\\
		& d_H(F^{k-2}(A_{ik+2}),F^{k-3}(A_{ik+3}))<\frac{\eta}{k+1}\\
		&\cdots\\
		& d_H(F^2(A_{ik++k-2}),F(A_{ik+k-1}))<\frac{\eta}{k+1}\\
		& d_H(F(A_{ik+k-1}),A_{ik+k})<\delta<\frac{\eta}{k+1}\\
	\end{aligned}
	\right.
\end{equation*}
Thus $$d_H(F^k(A_{ik}),A_{ik+k})<\frac{k\eta}{k+1}<\eta.$$ That is $\{A_{ik}\}_{i\geq 0}$ is a $\eta$-pseudo orbit of $F^k$. Then there exists $z\in X$ such that $$d_H(F^{nk}(z),A_{nk})<\gamma,\forall n\geq 0.$$
Since $\{A_{nk},A_{nk+1},\cdots,A_{nk+k+1},A_{nk+k}\}$ is a $\delta$-chain of $F$ and $\delta<\gamma$, then it is also a $\gamma$-chain of $F$ and $F^{nk}(z)\in B_{d_H}(A_{nk},\gamma)$. Thus $$d_H(F^{i+nk}(z),A_{i+nk})<\epsilon,\forall n\geq0, \forall 0\leq i\leq k.$$ That is, 
$$d_H(F^{n}(z),A_{n})<\epsilon,\forall n\geq0.$$ So, $F$ has shadowing property.
\end{proof}
\begin{theorem}
If $F$ has average shadowing property, then so does $F^k$ for any $k>1$
\end{theorem}
\begin{proof}
Let $k>1$ and $\epsilon>0$. Then there exists $\delta>0$ such that $F$ has average shadowing property witnessed by $\frac{\epsilon}{k}$. Let $\{B_i\}_{i\geq 0}$ be a $\delta$-average-pseudo-orbit of $F^k$. Then there exists $N=N(\delta)>0$ such that for any $n\geq N$ and any $m\in\mathbb{Z^+}$, $$\frac{1}{n}\sum^{n-1}_{i=0}d_H(F^k(A_{i+m}),A_{i+m+1})<\delta.$$
Let $A_{nk+j}=F^j(B_n)$, $\forall 0\leq j<k, n\geq 1$. Then $$\{A_i\}_{i\geq 0}=\{A_0,F(A_0),\cdots,F^{k-1}(A_0),A_1,F(A_1),\cdots,F^{k-1}(A_1),A_2,\cdots\}.$$
Then for any $n\geq N$ and any $m\in\mathbb{Z^+}$, $$\frac{1}{n}\sum^{n-1}_{i=0}d_H(F(A_{i+m}),A_{i+m+1})<\delta.$$
Hence, $\{A_i\}_{i\geq0}$ is a $\delta$-average-pseudo-orbit of $F$. Since $F$ has average shadowing property, there exists $z\in X$ such that 
\begin{equation}\label{eee1}
	\limsup_{n\rightarrow\infty}\frac{1}{n}\sum_{i=0}^{n-1}d_H(F^i(z),A_i)<\frac{\epsilon}{k}.
\end{equation}

Suppose that there exists $N_0\geq0$ such that $\frac{1}{n}\sum_{i=0}^{n-1}d_H(F^{ik}(z),A_{ik})\geq\epsilon$ for any $n\geq N_0$. Then 
$$\limsup_{n\rightarrow\infty}\frac{1}{n}\sum_{i=0}^{n-1}d_H(F^i(z),A_i)\geq\frac{\epsilon}{k},$$
which is a contradiction with (\ref{eee1}). Hence there exist infinitely many $n\in\mathbb{N}$ such that $$\frac{1}{n}\sum_{i=0}^{n-1}d_H(F^{ik}(z),A_{ik})<\epsilon.$$ Therefore, $$	\limsup_{n\rightarrow\infty}\frac{1}{n}\sum_{i=0}^{n-1}d_H(F^{ki}(z),B_{i})<\frac{\epsilon}{k}.$$ So, $F^k$ has average shadowing property.
%
%
%

\end{proof}
\section{Transitive properties}\label{sec5}
Abbas \cite{abbas5} claimed that if $f$ has shadowing property then $f$ is chain mixing is equivalent to $f$ being mixing. Alireza \cite{alireza} proved that if the iterated function systems with one of the continuous self-maps being surjective having average shadowing property then it is chain transitive. Now we show that there are similar conclusions for multiple mappings as well. 
\begin{theorem}
Suppose that $F$ has shadowing property. $F$ is chain mixing if and only if $F$ is mixing.
\end{theorem}
\begin{proof}
If $F$ is mixing, then it is obvious that $F$ is chain mixing. Let $F$ be chain mixing, then we show $F$ is mixing.

For any nonempty open sets $U\subset X$ and $\mathcal{U}\subset Ran(F)$, select $x\in U$, $A\in\mathcal{U}$ and $\epsilon>0$ satisfying $B_d(x,\epsilon)\subset U$ and $B_{d_H}(A,\epsilon)\subset\mathcal{U}$. Let $\delta>0$ be an $\epsilon$ modulus of shadowing property for $F$. Then there exists $N\in\mathbb{Z^+}$ such that for any $n\geq N$, there is a $\delta$-chain from $x$ to $A$ of length $n$. Since $F$ has shadowing property, one can find $z\in B_d(x,\epsilon)\subset U$ such that $F^n(z)\in B_{d_H}(A,\epsilon)\subset\mathcal{U}$. $\{z,F(z),F^2(z),\cdots,F^n(z)\}$ is an orbit of length $n$ begins in $U$ and ends in $\mathcal{U}$. So $F$ is mixing. 
\end{proof}

\begin{theorem}\label{Fac}
If $F$ has average shadowing property, then $F$ is chain transitive.
\end{theorem}
\begin{proof}
Let $\epsilon>0$, $x\in X$ and $A\in Ran(F)$. Since $F$ is continuous, there exists $0<\epsilon<\frac{\epsilon'}{2}$ such that $d(u,v)<2\epsilon'$ implies $d_H(F(u),F(v))<\epsilon$ for any $u,v\in X$ and $d_H(A,B)<2\epsilon'$ implies $d_H(F(A),F(B))<\epsilon$ for any $A,B\in\mathbb{K}(X)$. Select $\delta=\delta(\epsilon')$ such that $F$ has average shadowing property witnessed by $\epsilon'$. Take $N\in\mathbb{Z^+}$ such that $\frac{3D}{N}<\delta$ in which $D$ is the diameter of $X$. Then there exists $\{A_{N-2},A_{N-3},\cdots,A_1,A\}$ such that $F(A_i)=A_{i-1}$, $\forall N-2\leq i\leq 1$, where $A_0=A$. Let $$\{B_i\}_{i\geq0}=\{x,F(x),\cdots,F^N(x),A_{N-2},\cdots,A_1,A,x,F(x),\cdots,,A_1,A,x,,\cdots\}.$$ Then for any $n\geq N$ and any $k\in\mathbb{Z^+}$,$$\frac{1}{n}\sum_{i=0}^{n-1}d_H(F(B_{i+k}),B_{i+k+1})<\frac{1}{n}\frac{n}{N}3D\leq\frac{3D}{N}<\delta.$$
Thus, $\{B_i\}_{i\geq0}$ is a $\delta$-average-pseudo-orbit of $F$. Since $F$ has average shadowing property, there exists $z\in X$ such that $$\limsup_{n\rightarrow\infty}\frac{1}{n}\sum_{i=0}^{n-1}d_H(F^i(z),B_i)<\epsilon'.$$
Then there exist $i_0,l_0,j_1,j_2\in\mathbb{Z^+}$ with $i_0<l_0$ such that $$d_H(F^{i_0}(z),F^{j_1}(x))<2\epsilon'~and~d_H(F^{l_0}(z),A_{j_2})<2\epsilon'.$$ Thus $$\{x,F(x),\cdots,F^{j_1}(x),F^{i_0+1}(x),\cdots,F^{l_0-1}(x),A_{j_2},\cdots,A\}$$ is a $\epsilon$-chain from $x$ to $A$.
\end{proof}

Consider the multiple mappings $F$ of Example \ref{eab}. By Theorem \ref{Fac}, $F$ is chain transitive. It is easy to see that if $F$ is chain transitive and $F$ has shadowing property then $F$ is transitive. Then by Theorem \ref{Fac} we have the following corollary.
\begin{corollary}
If $F$ have both shadowing and average shadowing properties, then $F$ is transitive.
\end{corollary}

\section{Conclusions}\label{sec6}
This paper set out to study shadowing, average shadowing and transitive properties of multiple mappings from the perspective of a set-valued view. This perspective is different from the semigroup-theoretic view that has been previously studied in the context of dynamical systems of iterated function systems. We discussed the relationship between multiple mappings and its continuous self-maps in terms of the above properties.  It may not only deepen our understanding of continuous self-maps but also enable us to use relatively simple continuous self-maps to comprehend relatively complex multiple mappings. We show that for multiple mappings both shadowing and average shadowing properties are iterable, shadowing property plus chain mixing implies mixing and average shadowing properties implies chain transitivity. We hope that these conclusions can enrich the achievements in the field of fractal and provide some assistance for further in-depth research in this area. 
\section*{Acknowledgments}
We give our thanks to the referee for his (or her) careful reviewing which substantially improve the paper.  


\end{document}